\documentclass{amsart}
\usepackage{amsmath, amssymb, mathrsfs, verbatim}
\newcommand{\mathsym}[1]{{}}

\newtheorem{thm}{Theorem}[section]

\newtheorem{prop}[thm]{Proposition}
\newtheorem{crl}[thm]{Corollary}

\theoremstyle{definition}
\newtheorem{dfn}[thm]{Definition}
\newtheorem{exm}[thm]{Example}

\newtheorem{rem}[thm]{Remark}

\newcommand{\reals}{\mathbb{R}}
\newcommand{\naturals}{\mathbb{N}}
\newcommand{\ux}{\underline{X}}
\newcommand{\sgr}[2]{#1[#2]}
\newcommand{\rx}{\sgr{\mathbb{R}}{\ux}}
\newcommand{\sos}{\sum\mathbb{R}[\underline{X}]^2}
\newcommand{\SOBS}{SOBS}

\begin{document}

\title{lower bounds for polynomials using geometric programming}
 \date{}

\author{Mehdi Ghasemi}
\author{Murray Marshall}
\address{Department of Mathematics and Statistics,
University of Saskatchewan,
Saskatoon, \newline \indent
SK S7N 5E6, Canada}
\email{mehdi.ghasemi@usask.ca, marshall@math.usask.ca}


\begin{abstract}
We make use of a result of Hurwitz  and Reznick \cite{Hu} \cite{BR2},  and a consequence of this result due to Fidalgo and Kovacec \cite{FK},
to determine a new sufficient condition for a polynomial $f\in\mathbb{R}[X_1,\dots,X_n]$ of even degree to be a sum of squares. This result generalizes a result of Lasserre in \cite{JL} and a result of Fidalgo and Kovacec in \cite{FK}, and it also generalizes the improvements of these results given in \cite{MGHMM}. We apply this result to obtain a new lower bound $f_{gp}$ for $f$, and we explain how $f_{gp}$ can be computed using geometric programming. The lower bound $f_{gp}$ is generally not as good as the lower bound $f_{sos}$ introduced by Lasserre \cite{JLSDP} and Parrilo and Sturmfels \cite{PS}, which is computed using semidefinite programming, but a run time comparison shows that, in practice, the computation of $f_{gp}$ is much faster. The computation is simplest when the highest degree term of $f$ has the form $\sum_{i=1}^n a_iX_i^{2d}$, $a_i>0$, $i=1,\dots,n$.
The lower bounds for $f$ established in \cite{MGHMM} are obtained by evaluating the objective function of the geometric program at the appropriate feasible points.
\end{abstract}

\maketitle
%
\section{Introduction}
%
Fix a non-constant polynomial $f\in\rx=\reals[X_1,\cdots,X_n]$, where $n\ge1$ is an integer number, and let $f_*$ be the global minimum of $f$,
defined by
\[
	f_*:=\inf\{f(\underline{a})~:~\underline{a}\in\reals^n\}.
\]
We say $f$ is positive semidefinite (PSD) if $f(\underline{a})\ge0$ $\forall \underline{a}\in\reals^n$. Clearly
\[
	\inf\{ f(\underline{a}) : \underline{a} \in \reals^n\}=\sup\{r\in\reals~:~f-r\textrm{ is PSD}\},
\]
so finding $f_*$ reduces to determining when $f-r$ is PSD.

Suppose that $\deg(f)=m$ and decompose $f$ as $f=f_0+\cdots+f_m$ where $f_i$ is a form with $\deg(f_i)=i$, $i=0,\ldots,m$. This
decomposition is called the homogeneous decomposition of $f$. A necessary condition for $f_* \ne -\infty$ is that $f_m$ is PSD
(hence $m$ is even). A form $g\in\rx$ is said to be positive definite (PD) if $g(\underline{a})>0$ for all $\underline{a}\in\reals^n$, $\underline{a}\neq\underline{0}$.
A sufficient condition for $f_* \ne -\infty$ is that $f_m$ is PD \cite{M3}.

It is known that deciding when a polynomial is PSD is NP-hard \cite[Theorem 1.1]{CmplxtyNO}.
Deciding when
a polynomial is a sums of squares (SOS) is much easier. 
Actually, there is a polynomial time
method, known as semidefinite programming (SDP), which can be used to decide when a polynomial $f\in\rx$ is SOS \cite{JLSDP} \cite{PS}.
Note that any SOS polynomial is obviously PSD, so it is natural to ask if the converse is true, i.e. is every PSD polynomial SOS?
This question first appeared in Minkowski's thesis and he guessed that in general the answer is NO. Later, in \cite{Hil}, Hilbert gave a complete
answer to this question, see \cite[Section 6.6]{B-C-R}. Let us denote the cone of PSD forms of degree $2d$ in $n$ variables by $P_{2d,n}$
and the cone of SOS forms of degree $2d$ in $n$ variables by $\Sigma_{2d,n}$. Hilbert proved that
\textit{$P_{2d,n}=\Sigma_{2d,n}$ if and only if ($n\leq2$) or ($d=1$) or ($n=3$ and $d=2$)}.

Let $\sos$ denote the cone of all SOS polynomials in $\rx$ and, for $f\in\rx$, define
\[
	f_{sos}:=\sup\{r\in\reals~:~f-r\in\sos\}.
\]
Since SOS implies PSD, $f_{sos}\leq f_*$. Moreover, if $f_{sos}\neq-\infty$ then $f_{sos}$ can be computed in polynomial time, as close as
desired, using SDP \cite{JLSDP} \cite{PS}. We denote by $P_{2d,n}^{\circ}$ and $\Sigma_{2d,n}^{\circ}$, the interior of $P_{2d,n}$ and $\Sigma_{2d,n}$ in the vector space of forms of degree $2d$ in $\rx$, equipped with the euclidean topology.
A necessary condition for $f_{sos} \ne -\infty$ is that $f_{2d} \in \Sigma_{2d,n}$. A sufficient condition for $f_{sos} \ne -\infty$ is
that $f_{2d} \in \Sigma_{2d,n}^{\circ}$ \cite[Proposition. 5.1]{M2}.

In Section \ref{ScfSOS}, we recall the Hurwitz-Reznick result (Theorem \ref{HuRz}) and a corollary of the Hurwitz-Reznick result due to Fidalgo and Kovacek (Corollary \ref{FKkeythm}). For the convenience of the reader we include proofs of these results. Using the latter result, we determine a sufficient condition, in terms of the coefficients, for a form $f$ of degree $2d$ to be SOS (Theorem \ref{SuffCnd}). We explain how Theorem \ref{SuffCnd} can be applied to derive various concrete
criteria, in terms of the coefficients, for a form to be SOS, including results proved earlier by Lasserre \cite[Theorem 3]{JL}, Fidalgo and Kovacec \cite[Theorem 4.3]{FK}, and Ghasemi and Marshall \cite[Section 2]{MGHMM}. 

In Section \ref{AtGP}, we use Theorem \ref{SuffCnd} to establish a new lower bound $f_{gp}$ for $f$ and we explain how $f_{gp}$ can be computed using geometric programming.
An advantage of the method is that solving a geometric program is almost as fast as solving a linear program.
Although the lower bound found by this method is typically not as good as the lower bound found using SDP, a practical comparison shows that the computation is much faster, and larger problems can be handled.

In Section \ref{ExpBnds} we explain how results in Section \ref{AtGP} imply and improve on the results in  \cite[Section 3]{MGHMM}.

In this paper we denote by $\naturals$ the set of nonnegative integers $\{0,1,2,\ldots\}$. For $\ux=(X_1,\ldots,X_n)$, $\underline{a}=(a_1,\ldots,a_n)\in\reals^n$
and $\alpha=(\alpha_1,\ldots,\alpha_n)\in\mathbb{N}^n$,
define $\ux^{\alpha}:=X_1^{\alpha_1}\cdots X_n^{\alpha_n}$,
$|\alpha| := \alpha_1+\cdots+\alpha_n$ and $\underline{a}^{\alpha}:=a_1^{\alpha_1}\cdots a_n^{\alpha_n}$ with the convention $0^0=1$.
Clearly, using these notations, every polynomial $f\in\rx$ can be written as
$f(\ux)=\sum_{\alpha\in\naturals^n}f_{\alpha}\ux^{\alpha}$,
where $f_{\alpha}\in\reals$ and $f_{\alpha}=0$, except for finitely many $\alpha$. Assume now that $f$ is non-constant and has even degree.
Let $\Omega(f)=\{\alpha\in\naturals^n~:~f_{\alpha}\neq0\}\setminus\{\underline{0},2d\epsilon_1,\dots,2d\epsilon_n\}$, where
$2d=\deg(f)$, $\epsilon_i=(\delta_{i1},\dots,\delta_{in})$,  and
\[
\delta_{ij}=\left\lbrace
\begin{array}{lr}
	1 & i=j\\
	0 & i\neq j.
\end{array}\right.
\]
We denote $f_{\underline{0}}$ and $f_{2d\epsilon_i}$ by $f_0$ and $f_{2d,i}$ for short. Thus $f$
has the form
\begin{equation}\label{polyformat}
f=f_0+\sum_{\alpha\in\Omega(f)}f_{\alpha}\ux^{\alpha}+\sum_{i=1}^nf_{2d,i}X_i^{2d}.
\end{equation}
Let $\Delta(f)=\{\alpha\in\Omega(f):~f_{\alpha}\ux^{\alpha}\text{ is not a square in }\rx\}=\{\alpha\in\Omega(f):\text{either } f_{\alpha}<0\text{ or }\alpha_i\text{ is odd for some } 1\leq i\leq n\}$. Since the polynomial $f$ is usually fixed, we will often
denote $\Omega(f)$ and $\Delta(f)$ just by $\Omega$ and $\Delta$ for short.

Let $\bar{f}(\ux,Y)=Y^{2d}f(\frac{X_1}{Y},\ldots,\frac{X_n}{Y})$. From \eqref{polyformat} it is clear that
\[
\bar{f}(\ux,Y)=f_0Y^{2d}+\sum_{\alpha\in\Omega}f_{\alpha}\ux^{\alpha}Y^{2d-|\alpha|}+\sum_{i=1}^nf_{2d,i}X_i^{2d}
\]
is a form of degree $2d$, called the homogenization of $f$. We have the following well-known result:
\begin{prop}
	\label{homodehomo}
$f$ is PSD if and only if $\bar{f}$ is PSD. $f$ is SOS if and only if $\bar{f}$ is SOS.
\end{prop}
\begin{proof}
See \cite[Proposition 1.2.4]{M1}.
\end{proof}
%
\section{Sufficient conditions for a form to be SOS}\label{ScfSOS}
%
We recall the following result, due to Hurwitz and Reznick. 
\begin{thm}[Hurwitz-Reznick]\label{HuRz}
 Suppose $p(\ux)=\sum_{i=1}^n\alpha_iX_i^{2d}-2dX_1^{\alpha_1}\cdots X_n^{\alpha_n}$,
where $\alpha=(\alpha_1,\ldots,\alpha_n)\in\mathbb{N}^n$,
 $|\alpha|=2d$. Then $p$ is \SOBS.
\end{thm}
Here, \SOBS~ is shorthand for a sum of binomial squares, i.e., a sum of squares of the form  $(a\ux^{\alpha}-b\ux^{\beta})^2$

In his 1891 paper \cite{Hu}, Hurwitz  uses symmetric polynomials in $X_1,\dots,X_{2d}$ to give an explicit representation of
$\sum_{i=1}^{2d}X_i^{2d}-2d\prod_{i=1}^{2d}X_i$ as a sum of squares. Theorem \ref{HuRz} can be deduced from this representation.
Theorem \ref{HuRz} can also be deduced from results in \cite{BR1,BR2}, specially, from \cite[Theorems 2.2 and 4.4]{BR2}. Here is another proof.

\begin{proof}
By induction on $n$. If $n = 1$ then $p = 0$ and the result is clear. Assume now that
$n\ge2$. We can assume each $\alpha_i$ is strictly positive, otherwise, we reduce to a case with at most $n-1$ variables.

Case 1: Suppose that there exist $1\leq i_1,i_2\leq n$, such that $i_1\neq i_2$, with $\alpha_{i_1}\le d$ and $\alpha_{i_2}\le d$.
Decompose $\alpha=(\alpha_1,\ldots,\alpha_n)$ as $\alpha=\beta + \gamma$ where $\beta,\gamma\in\mathbb{N}^n$,
$\beta_{i_1}=0$, $\gamma_{i_2} = 0$ and $|\beta|=|\gamma|= d$. Then
\[
	 (\ux^{\beta}-\ux^{\gamma})^2=\ux^{2\beta}-2\ux^{\beta}\ux^{\gamma}+\ux^{2\gamma}=\ux^{2\beta}-2\ux^{\alpha}+\ux^{2\gamma},
\]
therefore,
\[
\begin{array}{lll}
	p(\ux) & = & \displaystyle{\sum_{i=1}^n \alpha_iX_i^{2d}-2d\ux^{\alpha}}\\
	& = & \displaystyle{\sum_{i=1}^n \alpha_iX_i^{2d}-d(\ux^{2\beta}+\ux^{2\gamma}-(\ux^{\beta}-\ux^{\gamma})^2)}\\
	& = & \displaystyle{\frac{1}{2}\left(\sum_{i=1}^n 2\beta_iX_i^{2d}-2d\ux^{2\beta}\right)}\\
	& + & \displaystyle{\frac{1}{2}\left(\sum_{i=1}^n 2\gamma_iX_i^2d-2d\ux^{2\gamma}\right)+d(\ux^{\beta}-\ux^{\gamma})^2}.
\end{array}
\]
Each term is \SOBS, by induction hypothesis.

Case 2: Suppose we are not in Case 1. Since there is at most one $i$ satisfying $\alpha_i > d$, it
follows that $n = 2$, so $p(\ux) = \alpha_1 X_1^{2d} + \alpha_2 X_2^{2d}-2dX_1^{\alpha_1} X_2^{\alpha_2}$. We know that $p \ge 0$ on
$\mathbb{R}^2$, by the
arithmetic-geometric inequality. Since $n = 2$ and $p$ is homogeneous, it follows that $p$ is SOS.

Showing $p$ is \SOBS, requires more work. Denote by $\textrm{AGI}(2, d)$ the
set of all homogeneous polynomials of the form $p = \alpha_1 X_1^{2d} + \alpha_2 X_2^{2d} - 2dX_1^{\alpha_1} X_2^{\alpha_2}$ ,
$\alpha_1, \alpha_2\in\mathbb{N}$ and
$\alpha_1 + \alpha_2 = 2d$. This set is finite. If $\alpha_1 = 0$ or $\alpha_1 = 2d$ then $p = 0$ which is trivially \SOBS.
If $\alpha_1 = \alpha_2 = d$ then $p(\ux) = d(X_1^d-X_2^d)^2$, which is also \SOBS. Suppose now that $0 < \alpha_1 < 2d$,
$\alpha_1 \neq d$ and $\alpha_1 > \alpha_2$ (The argument for $\alpha_1 < \alpha_2$ is similar). Decompose
$\alpha = (\alpha_1 , \alpha_2 )$ as $\alpha = \beta + \gamma$, $\beta = (d, 0)$ and $\gamma = (\alpha_1 - d, \alpha_2 )$.
Expand $p$ as in the proof of Case 1 to obtain
\[
p(\ux)=\frac{1}{2}\left(\sum_{i=1}^22\beta_iX_i^{2d}-2d\ux^{2\beta}\right)+
\frac{1}{2}\left(\sum_{i=1}^22\gamma_iX_i^{2d}-2d\ux^{2\gamma}\right)+d(\ux^{\beta}-\ux^{\gamma})^2.
\]
Observe that $\sum_{i=1}^2 2\beta_i X_i^{2d}-2d\ux^{2\beta} = 0$.

Thus $p = \frac{1}{2}p_1 +d(\ux^{\beta}-\ux^{\gamma} )^2$, where
$p_1=\sum_{i=1}^22\gamma_iX_i^{2d}-2d\ux^{2\gamma}$. If $p_1$ is \SOBS~ then $p$ is also \SOBS. If $p_1$ is not \SOBS~ then we can repeat to
get $p_1=\frac{1}{2}p_2+d(\ux^{\beta'}-\ux^{\gamma'})^2$.
Continuing in this way we get a sequence $p = p_0 , p_1 , p_2 ,\cdots$ with each
$p_i$ an element of the finite set $\textrm{AGI}(2, d)$, so $p_i = p_j$ for some $i < j$. Since $p_i = 2^{i-j} p_j +$ \textit{a
sum of binomial squares}, this implies $p_i$ is \SOBS~ and hence that $p$ is \SOBS.
\end{proof}
In \cite{FK}, Fidalgo and Kovacec prove the following result, which is a corollary of the Hurwitz-Reznick result.
\begin{crl}[Fidalgo-Kovacek]\label{FKkeythm}
 For a form $p(\ux)=\sum_{i=1}^n\beta_iX_i^{2d}-\mu\ux^{\alpha}$ such that $\alpha \in \mathbb{N}^n$, $|\alpha| =2d$,
$\beta_i\ge 0$ for $i=1,\cdots,n$, and $\mu\ge0$ if all $\alpha_i$ are even, the following are equivalent:
\begin{enumerate}
	\item{$p$ is PSD.}
	\item{$\mu^{2d}\prod_{i=1}^n\alpha_i^{\alpha_i}\leq (2d)^{2d}\prod_{i=1}^n\beta_i^{\alpha_i}$.}
	\item{$p$ is \SOBS.}
	\item{$p$ is SOS.}
\end{enumerate}
\end{crl}

\begin{proof} See \cite[Theorem 2.3]{FK}. (3) $\Rightarrow$ (4) and (4) $\Rightarrow$ (1) are trivial, so it suffices to show (1) $\Rightarrow$ (2) and (2) $\Rightarrow$ (3). If some $\alpha_i$ is odd then, making the change of variables $Y_i=-X_i$, $Y_j=X_j$ for $j\ne i$, $\mu$ gets replaced by $-\mu$. In this way, we can assume $\mu \ge 0$. If some $\alpha_i$ is zero, set $X_i=0$ and proceed by induction on $n$. In this way, we can assume $\alpha_i>0$, $i=1,\dots,n$. If $\mu=0$ the result is trivially true, so we can assume $\mu >0$. If some $\beta_i$ is zero, then (2) fails. Setting $X_j=1$ for $j\ne i$, and letting $X_i \rightarrow \infty$, we see that (1) also fails. Thus the result is trivially true in this case. Thus we can assume $\beta_i>0$, $i=1,\dots,n$. 

(1) $\Rightarrow$ (2). Assume (1), so $p(x) \ge 0$ for all $x\in \mathbb{R}^n$. Taking $$x := ( (\frac{\alpha_i}{\beta_i})^{1/2d},\dots,(\frac{\alpha_n}{\beta_n})^{1/2d}),$$ we see that $$p(x) = \sum_{i=1}^n \alpha_i -\mu\prod_{i=1}^n (\frac{\alpha_i}{\beta_i})^{\alpha_i/2d} = 2d -\mu\prod_{i=1}^n (\frac{\alpha_i}{\beta_i})^{\alpha_i/2d} \ge 0,$$ so $\mu \prod_{i=1}^n (\frac{\alpha_i}{\beta_i})^{\alpha_i/2d} \le 2d$. This proves (2).

(2) $\Rightarrow$ (3). Make a change of variables $X_i = (\frac{\alpha_i}{\beta_i})^{1/2d} Y_i$, $i=1,\dots n$. Let $\mu_1 := \mu \prod_{i=1}^n (\frac{\alpha_i}{\beta_i})^{\alpha_i/2d}$ so, by (2), $\mu_1 \le 2d$, i.e., $\frac{2d}{\mu_1}\ge 1$. Then $$p(\underline{X}) = \sum_{i=1}^n \alpha_iY_i^{2d}-\mu_1\underline{Y}^{\alpha} = \frac{\mu_1}{2d}[\sum_{i=1}^n \alpha_iY_i^{2d}(\frac{2d}{\mu_1}-1)+\sum_{i=1}^n \alpha_i Y_i^{2d}-2d\underline{Y}^{\alpha}],$$ which is SOBS, by the Hurwitz-Reznick result. This proves (3).
\end{proof}

Next, we prove our main new result of this section,
which gives a sufficient condition on the coefficients for a polynomial to be a sum of
squares.
\begin{thm}\label{SuffCnd}
Suppose $f$ is a form of degree $2d$. A sufficient condition for $f$ to be \SOBS~ is that there exist nonnegative real numbers $a_{\alpha,i}$
for $\alpha\in\Delta$, $i=1,\ldots,n$ such that
\begin{enumerate}
	 \item{$\forall\alpha\in\Delta\quad(2d)^{2d}a_{\alpha}^{\alpha}=f_{\alpha}^{2d}\alpha^{\alpha}$.}
	\item{$f_{2d,i}\ge\sum_{\alpha\in\Delta}a_{\alpha,i}$, $i=1,\ldots,n$.}
\end{enumerate}
Here, $a_{\alpha}:=(a_{\alpha,1},\dots,a_{\alpha,n})$.
\end{thm}
\begin{proof}
Suppose that such real numbers exist. Then condition (1) together with Corollary \ref{FKkeythm} implies that
$\sum_{i=1}^n a_{\alpha,i}X_i^{2d} + f_{\alpha}\ux^{\alpha}$ is \SOBS~ for each $\alpha\in\Delta$, so
\[
	 \sum_{i=1}^n(\sum_{\alpha\in\Delta}a_{\alpha,i})X_i^{2d}+\sum_{\alpha\in\Delta}f_{\alpha}\ux^{\alpha}
\]
is \SOBS. Combining with (2), it follows that $\sum_{i=1}^n f_{2d,i}X_i^{2d}+\sum_{\alpha\in\Delta}f_{\alpha}\ux^{\alpha}$
is \SOBS. Since each $f_{\alpha}\ux^{\alpha}$ for $\alpha\in\Omega\setminus \Delta$ is a
square, this implies $f(\ux)$ is \SOBS.
\end{proof}
\begin{rem}

(i) From condition (1) of Theorem \ref{SuffCnd} we see that $a_{\alpha,i}=0$ $\Rightarrow$ $\alpha_i=0$.
(ii) Let $a$ be an array of real numbers satisfying the conditions of Theorem \ref{SuffCnd}, and define the array $a^*=(a_{\alpha,i}^*)$ by
\[
	a_{\alpha,i}^*=\left\lbrace\begin{array}{ll}
		a_{\alpha,i} & \textrm{if } \alpha_i\neq 0\\
		0 & \textrm{if } \alpha_i=0.
	\end{array}\right.
\]
Then $a^*$ also satisfies the conditions of Theorem \ref{SuffCnd}. Thus we are free to require the converse condition $\alpha_i = 0$ $\Rightarrow$ $a_{\alpha,i} = 0$ too, if we want.
\end{rem}

We mention 
some corollaries of Theorem \ref{SuffCnd}. Corollaries \ref{JLsos} and \ref{newfk} were known earlier. Corollary \ref{fkimp} is an improved version of Corollary \ref{newfk}. Corollary \ref{newcrt} is a new result.

\begin{crl}See \cite[Theorem 3]{JL} and \cite[Theorem 2.2]{MGHMM}.\label{JLsos}
For any polynomial $f\in \mathbb{R}[\underline{X}]$ of degree $2d$, if

\medskip
{\rm (L1)}  \ \ $f_0\ge \sum\limits_{\alpha\in \Delta} |f_{\alpha}| \frac{2d-|\alpha|}{2d}$  \ and   \ {\rm (L2)}  \ \ $f_{2d,i} \ge \sum\limits_{\alpha\in \Delta} |f_{\alpha}| \frac{\alpha_i}{2d}$, \ \ $i=1,\dots,n$,

\smallskip
\noindent
then $f$ is a sum of squares.
\end{crl}
\begin{proof}
Apply Theorem \ref{SuffCnd} to the homogenization $\bar{f}(\ux,Y)$ of $f$, taking $a_{\alpha,i}=|f_{\alpha}|\frac{\alpha_i}{2d}$, $i=1,\ldots,n$ and
$a_{\alpha,Y}=|f_{\alpha}|\frac{2d-|\alpha_i|}{2d}$ for each $\alpha \in \Delta$.
For $\alpha\in\Delta$,
\[
\begin{array}{lll}
	(2d)^{2d}a_{\alpha}^{\alpha} & = & (2d)^{2d}\left(\frac{|f_{\alpha}|(2d-|\alpha|)}{2d}\right)^{2d-|\alpha|}\prod_{i=1}^n\left(\frac{|f_{\alpha}|\alpha_i}{2d}\right)^{\alpha_i}\\
	& = & (2d)^{2d}|f_{\alpha}|^{2d-|\alpha|}(2d-|\alpha|)^{2d-|\alpha|}|f_{\alpha}|^{|\alpha|}\alpha^{\alpha}(2d)^{-2d}\\
	& = & |f_{\alpha}|^{2d}\alpha^{\alpha}(2d-|\alpha|)^{2d-|\alpha|}.
\end{array}
\]
So, \ref{SuffCnd}(1) holds. (L1) and (L2) imply \ref{SuffCnd}(2), therefore, by Theorem \ref{SuffCnd}, $\bar{f}$ and hence $f$ is \SOBS.
\end{proof}
\begin{crl}See \cite[Theorem 4.3]{FK} and \cite[Theorem 2.3]{MGHMM}.\label{newfk} Suppose $f\in\rx$ is a form of degree $2d$ and
$$\min\limits_{i=1,\ldots,n}f_{2d,i} \ge \frac{1}{2d}\sum\limits_{\alpha\in\Delta}|f_{\alpha}|(\alpha^{\alpha})^{\frac{1}{2d}}.$$
Then $f$ is \SOBS.
\end{crl}
\begin{proof}
Apply Theorem \ref{SuffCnd} with $a_{\alpha,i}=|f_{\alpha}|\frac{\alpha^{\alpha/2d}}{2d}$, $\forall\alpha\in\Delta$, $i=1\ldots,n$.
\end{proof}
\begin{crl}\label{fkimp}
Suppose $f$ is a form of degree $2d$, $f_{2d,i}>0$, $i=1,\dots,n$ and $$\sum_{\alpha\in\Delta}\frac{|f_{\alpha}|\alpha^{\alpha/2d}}{2d\prod_{i=1}^nf_{2d,i}^{\alpha_i/2d}}\leq 1.$$ Then $f$ is \SOBS.
\end{crl}
\begin{proof}
Apply Theorem \ref{SuffCnd} with $a_{\alpha,i}=\frac{|f_{\alpha}|\alpha^{\alpha/2d}f_{2d,i}}{2d\prod_{j=1}^nf_{2d,j}^{\alpha_j/2d}}$.
\end{proof}
\begin{rem}
Corollary \ref{fkimp} is an improved version of Corollary \ref{newfk}. This requires some explanation. Suppose that $f_{2d,i} \ge \frac{1}{2d}\sum_{\alpha\in \Delta} |f_{\alpha}|\alpha^{\alpha/2d}$, $i=1,\dots,n$. Let $f_{2d,i_0} := \min\{ f_{2d,i} : i=1,\dots,n\}$. Then $$\prod_{i=1}^n f_{2d,i}^{\alpha_i/2d} \ge \prod_{i=1}^n f_{2d,i_0}^{\alpha_i/2d} = f_{2d,i_0},$$ and
\begin{align*}
\sum_{\alpha\in \Delta}\frac{|f_{\alpha}|\alpha^{\alpha/2d}}{2d\prod_{i=1}^nf_{2d,i}^{\alpha_i/2d}} =& \frac{1}{2d} \sum_{\alpha\in \Delta} \frac{|f_{\alpha}|\alpha^{\alpha/2d}}{f_{2d,i_0}} \frac{f_{2d,i_0}}{\prod_{i=1}^n f_{2d,i}^{\alpha_i/2d}} \\ \le& \frac{1}{2d}\sum_{\alpha\in \Delta} \frac{|f_{\alpha}|\alpha^{\alpha/2d}}{f_{2d,i_0}}\le 1.
\end{align*}
\end{rem}
We note yet another sufficient condition for SOS-ness.
\begin{crl}\label{newcrt}
Let $f\in\rx$ be a form of degree $2d$. If
\[
	f_{2d,i}\ge \sum_{\alpha\in\Delta, \alpha_i\ne 0}\alpha_i\left(\frac{|f_{\alpha}|}{2d}\right)^{2d/\alpha_in_{\alpha}}, \ =1,\dots,n
\]
then $f$ is \SOBS. Here $n_{\alpha}:=|\{i:\alpha_i\neq0\}|$.
\end{crl}
\begin{proof}
Apply Theorem \ref{SuffCnd} with
\[
	 a_{\alpha,i}=\left\lbrace\begin{array}{ll}\alpha_i\left(\frac{|f_{\alpha}|}{2d}\right)^{2d/\alpha_in_{\alpha}} & \textrm{if } \alpha_i\neq 0\\
		0 & \textrm{if } \alpha_i=0.
	\end{array}\right.
\]
\end{proof}
The following example shows that the above corollaries are not as strong, either individually or collectively, as Theorem \ref{SuffCnd} itself.
\begin{exm}\label{deg6}
Let $f(X,Y,Z)=X^6+Y^6+Z^6-5X-4Y-Z+8$.
Corollary \ref{JLsos} does not apply to $f$, actually (L1) fails.
Also, Corollaries \ref{newfk}, \ref{fkimp} and \ref{newcrt} do not apply to $\bar{f}$, the homogenization of $f$.
We try to apply Theorem \ref{SuffCnd}. Let $\alpha_1=(1,0,0,5)$, $\alpha_2=(0,1,0,5)$ and  $\alpha_3=(0,0,1,5)$, then
$\Delta=\{\alpha_1, \alpha_2, \alpha_3\}$. Denote $a_{\alpha_i,j}$ by $a_{ij}$, we have to find positive reals
$a_{11}, a_{22}, a_{33}, a_{14}, a_{24}, a_{34}$ such that the followings hold:
\[
\begin{array}{ll}
	6^6a_{11}a_{14}^5 = 5^65^5, & 1\ge a_{11},\\
	6^6a_{22}a_{24}^5 = 4^65^5, & 1\ge a_{22},\\
	6^6a_{33}a_{34}^5 = 5^5, & 1\ge a_{33},\\
	8\ge a_{14}+a_{24}+a_{34}.&
\end{array}
\]
Take $a_{11}=a_{22}=a_{33}=1$ and solve equations on above set of conditions, we get $a_{14}+a_{24}+a_{34}\approx 7.674 < 8$. This
implies that $\bar{f}$ and hence $f$ is \SOBS.
\end{exm}

\section{Application to global optimization}\label{AtGP}
%
Let $f\in \rx$ be a non-constant polynomial of degree $2d$. Recall that $f_{sos}$ denotes the supremum of all real numbers $r$ such that $f-r\in\sos$, $f_*$ denotes the infimum of the set $\{ f(\underline{a}) : \underline{a} \in \mathbb{R}^n\}$, and $f_{sos}\leq f_*$.

Suppose $\underline{f}$ denotes the array of coefficients of non-constant terms of $f$ and $f_0$ denotes the constant term of $f$.
Suppose $\Phi(\underline{f},f_0)$ is a formula in terms of coefficients of $f$ such that
$\Phi(\underline{f},f_0)$ implies $f$ is SOS.
For such a criterion $\Phi$, we have
\[
\forall r~(\Phi(\underline{f},f_0-r)\rightarrow ~r\leq f_{sos}),
\]
so $f_{\Phi} :=\sup\{r\in\mathbb{R} : \Phi(\underline{f},f_0-r)\}$ is a lower bound for $f_{sos}$ and, consequently, for $f_*$.
In this section we develop this idea, using Theorem \ref{SuffCnd}, to find a new lower bound for $f$.
\begin{thm}\label{fgp}
Let $f$ be a non-constant polynomial of degree $2d$ and $r\in\reals$. Suppose there exist nonnegative real numbers $a_{\alpha,i}$, $\alpha\in\Delta$,
$i=1,\ldots,n$, $a_{\alpha,i} = 0$ iff $\alpha_i=0$, such that
\begin{enumerate}
	\item{$(2d)^{2d}a_{\alpha}^{\alpha}=|f_{\alpha}|^{2d}\alpha^{\alpha}$ for each $\alpha\in\Delta$ such that $|\alpha|=2d$,}
	\item{$f_{2d,i}\ge\sum_{\alpha\in\Delta}a_{\alpha,i}$ for $i=1,\ldots,n$, and}
	 \item{$f_0-r\ge\sum_{\alpha\in\Delta^{<2d}}(2d-|\alpha|)\left[\frac{|f_{\alpha}|^{2d}\alpha^{\alpha}}{(2d)^{2d}a_{\alpha}^{\alpha}}\right]^{\frac{1}{2d-|\alpha|}}$.}
\end{enumerate}
Then $f-r$ is \SOBS.
Here $\Delta^{<2d}:=\{\alpha\in\Delta:|\alpha|<2d\}$.
\end{thm}
\begin{proof}
Apply Theorem \ref{SuffCnd} to $g:=\overline{f-r}$, the homogenization of $f-r$. Since
$f=f_0+\sum_{i=1}^nf_{2d,i}X_i^{2d}+\sum_{\alpha\in\Omega}f_{\alpha}\ux^{\alpha}$, it follows that
$g=(f_0-r)Y^{2d}+\sum_{i=1}^nX_i^{2d}+\sum_{\alpha\in\Omega}f_{\alpha}\ux^{\alpha}Y^{2d-|\alpha|}$. We know $f-r$ is \SOBS~ if and only
if $g$ is \SOBS. The sufficient condition for $g$ to be \SOBS~ given by Theorem \ref{SuffCnd} is that there exist non-negative real numbers
$a_{\alpha,i}$ and $a_{\alpha,Y}$, $a_{\alpha,i} = 0$ iff $\alpha_i=0$, $a_{\alpha,Y}=0$ iff $|\alpha|=2d$ such that\\
\indent$(1)'$ $\forall$ $\alpha \in \Delta$ $(2d)^{2d}a_{\alpha}^{\alpha}a_{\alpha,Y}^{2d-|\alpha|}=|f_{\alpha}|^{2d}\alpha^{\alpha}(2d-|\alpha|)^{2d-|\alpha|}$,
and\\
\indent$(2)'$ $f_{2d,i}\ge\sum_{\alpha\in\Delta}a_{\alpha,i}$, $i=1,\ldots,n$ and $f_0-r\ge\sum_{\alpha\in\Delta}a_{\alpha,Y}$.\\
Solving $(1)'$ for $a_{\alpha,Y}$ yields
\[
a_{\alpha,Y}=(2d-|\alpha|)\left[\frac{|f_{\alpha}|^{2d}\alpha^{\alpha}}{(2d)^{2d}a_{\alpha}^{\alpha}}\right]^{\frac{1}{2d-|\alpha|}},
\]
if $|\alpha|<2d$. Take $a_{\alpha,Y}=0$ if $|\alpha|=2d$. Conversely, defining $a_{\alpha,Y}$ in this way, for each $\alpha \in \Delta$, it is easy to see that (1), (2), and (3) imply $(1)'$ and $(2)'$.
\end{proof}
\begin{dfn}
For a non-constant polynomial $f$ of degree $2d$ we define
\begin{align*}
f_{gp}:=\sup\{ &r\in\reals: \exists a_{\alpha,i}\in\reals^{\ge0},~\alpha\in\Delta, i=1,\dots,n, \ a_{\alpha,i}=0 \textrm{ iff } \alpha_i =0 \\ & \textrm{ satisfying conditions (1), (2) and (3) of Theorem \ref{fgp}}\}.
\end{align*}

\end{dfn}
It follows, as a consequence of Theorem \ref{fgp}, that $f_{gp}\le f_{sos}$.

\begin{exm} Let $f(X,Y) = X^4+Y^4-X^2Y^2+X+Y$. Here, $\Delta = \{ \alpha_1,\alpha_2,\alpha_3\}$, where $\alpha_1=(1,0)$, $\alpha_2=(0,1)$ and $\alpha_3=(2,2)$. We are looking for non-negative reals $a_{i,j}$, $i=1,2,3$, $j=1,2$ satisfying $a_{11}+a_{21}+a_{31}\le1$, $a_{12}+a_{22}+a_{32}\le 1$, $a_{31}a_{32} = \frac{1}{4}$. Taking $a_{11}=a_{22}=a_{31}=a_{32}= \frac{1}{2}$, $a_{12}=a_{21}=0$, we see that $f_{gp} \ge -\frac{3}{2^{4/3}}$. Taking $X= Y = -\frac{1}{2^{1/3}}$ we see that $f_* \le f(-\frac{1}{2^{1/3}}, -\frac{1}{2^{1/3}}) = -\frac{3}{2^{4/3}}$. Since $f_{gp} \le f_{sos}\le f_*$, it follows that $f_{gp}=f_{sos}=f_*= -\frac{3}{2^{4/3}}$.
\end{exm}
\begin{rem}\label{singleterm}
If $|\Omega|=1$ then $f_*=f_{sos}=f_{gp}$.

\begin{proof} Say $\Omega = \{ \alpha\}$, so $f = \sum_{i=0}^n f_{2d,i}X_i^{2d}+f_0+f_{\alpha}\underline{X}^{\alpha}$. We know $f_{gp} \le f_{sos}\le f_*$, so it suffices to show that, for each real number $r$, $f_*\ge r$ $\Rightarrow$ $f_{gp}\ge r$. Fix $r$ and assume $f_*\ge r$. We want to show $f_{gp}\ge r$, i.e., that $r$ satisfies the constrains of Theorem \ref{fgp}. Let $g$ denote the homogenization of $f-r$, i.e., $g = \sum_{i=1}^n f_{2d,i}X_i^{2d}+(f_0-r)Y^{2d}+f_{\alpha}\underline{X}^{\alpha}Y^{2d-|\alpha|}$. Thus $g$ is PSD. This implies, in particular, that $f_{2d,i}\ge 0$, $i=1,\dots,n$ and $f_0\ge r$. There are two cases to consider.

Case 1. Suppose $f_{\alpha}>0$ and all $\alpha_i$ are even. Then $\alpha\notin \Delta$, so $\Delta = \emptyset$. In this case $r$ satisfies trivially the constraints of Theorem \ref{fgp}, so $f_{gp}\ge r$.

Case 2. Suppose either $f_{\alpha}<0$ or not all of the $\alpha_i$ are even. Then $\alpha \in \Delta$, i.e., $\Delta = \Omega = \{ \alpha\}$. In this case, applying Corollary \ref{FKkeythm}, we deduce that
\begin{equation}\label{*}
f_{\alpha}^{2d}\alpha^{\alpha}(2d-|\alpha|)^{2d-|\alpha|} \le (2d)^{2d} \prod_{i=1}^n f_{2d,i}^{\alpha_i}(f_0-r)^{2d-|\alpha|}.
\end{equation}
There are two subcases to consider. If $|\alpha|< 2d$ then $r$ satisfies the constraints of Theorem \ref{fgp}, taking
\[
	a_{\alpha,i}=\left\lbrace\begin{array}{ll}
		f_{2d,i} & \textrm{if } \alpha_i\neq 0\\
		0 & \textrm{if } \alpha_i=0.
	\end{array}\right.
\]
If $|\alpha|=2d$ then (\ref{*}) reduces to $f_{\alpha}^{2d}\alpha^{\alpha} \le (2d)^{2d} \prod_{i=1}^n f_{2d,i}^{\alpha_i}$. In this case, $r$ satisfies the constraints of Theorem \ref{fgp}, taking
\[
	a_{\alpha,i}=\left\lbrace\begin{array}{ll}
		sf_{2d,i} & \textrm{if } \alpha_i\neq 0\\
		0 & \textrm{if } \alpha_i=0.
	\end{array}\right.
\]
where $$s= \left[\frac{|f_{\alpha}|^{2d}\alpha^{\alpha}}{(2d)^{2d}\prod_{i=1}^n f_{2d,i}^{\alpha_i}}\right]^{\frac{1}{|\alpha|}}.$$
\end{proof}
\end{rem}

If $f_{2d,i}>0$, $i=1,\ldots,n$ then computation of $f_{gp}$ is a geometric programming problem. We explain this now.
\begin{dfn}(geometric program)

(1) A function $f:\reals_{>0}^n\rightarrow\reals$ of the form
\[
	\phi(\underline{x})=cx_1^{a_1}\cdots x_n^{a_n},
\]
where $c>0$, $a_i\in\reals$ and $\underline{x}=(x_1,\ldots,x_n)$ is called a \textit{monomial function}. A sum of monomial functions, i.e., a function of the form
\[
	\phi(\underline{x})=\sum_{i=1}^k c_i x_1^{a_{1i}}\cdots x_n^{a_{ni}}
\]
where $c_i>0$ for $i=1,\dots,k$, is called a \textit{posynomial function}.\\

(2) An optimization problem of the form
\[
	\left\lbrace
	\begin{array}{ll}
		\textrm{Minimize} & \phi_0(\underline{x})  \\
		\textrm{Subject to} & \phi_i(\underline{x})\leq 1, \ i=1,\ldots,m \text{ and } \  \psi_i(\underline{x})=1, \ i=1,\ldots,p
	\end{array}
	\right.
\]
where $\phi_0,\ldots,\phi_m$ are posynomials and $\psi_1,\ldots,\psi_p$ are monomial functions, is called a \textit{geometric program}.
\end{dfn}
See \cite[Section 4.5]{SB-LV} or \cite[Section 5.3]{NLPpsu} for detail on geometric programs.
\begin{crl}\label{fgprog}
Let $f$ be a non-constant polynomial of degree $2d$ with $f_{2d,i}>0$, $i=1,\dots,n$. Then $f_{gp} = f_0-m^*$ where $m^*$ is the output of the geometric program
\[
\left\lbrace
\begin{array}{ll}
	\textrm{Minimize} & \sum_{\alpha\in\Delta^{<2d}}(2d-|\alpha|)\left[\left(\frac{f_{\alpha}}{2d}\right)^{2d}\alpha^{\alpha}a_{\alpha}^{-\alpha}\right]^{\frac{1}{2d-|\alpha|}} \\
	\textrm{Subject to} & \sum_{\alpha\in\Delta}\frac{a_{\alpha,i}}{ f_{2d,i}} \le 1, \  i=1,\cdots,n \ \text{ and } \ \ \frac{(2d)^{2d}a_{\alpha}^{\alpha}}{|f_{\alpha}|^{2d}\alpha^{\alpha}} =1, \ \alpha \in \Delta, \ |\alpha|=2d.\\
\end{array}
\right.
\]
The variables in the program are the $a_{\alpha,i}$, $\alpha\in \Delta$, $i=1,\dots,n$, $\alpha_i\ne 0$, the understanding being that $a_{\alpha,i} =0$ iff $\alpha_i=0$.
\end{crl}
\begin{proof}
$f_{gp} = f_0-m^*$ is immediate from the definition of $f_{gp}$. Observe that $$\phi_0(a):=\sum_{\alpha\in\Delta, |\alpha|<2d}(2d-|\alpha|)\left[\left(\frac{f_{\alpha}}{2d}\right)^{2d}\alpha^{\alpha}a_{\alpha}^{-\alpha}\right]^{\frac{1}{2d-|\alpha|}}$$ and $\phi_i(a):= \sum_{\alpha\in\Delta}\frac{a_{\alpha,i}}{ f_{2d,i}}, \ i=1,\dots,n$ are posynomials in the variables $a_{\alpha,i}$, and $\psi_{\alpha}(a) := \frac{(2d)^{2d}a_{\alpha}^{\alpha}}{|f_{\alpha}|^{2d}\alpha^{\alpha}}$, $\alpha \in \Delta$, $|\alpha|=2d$ are monomial functions in the variables $a_{\alpha,i}$.
\end{proof}
\noindent
\bf Addendum: \rm If either $f_{2d,i}<0$ for some $i$ or $f_{2d,i}=0$ and $\alpha_i \ne 0$ for some $i$ and some $\alpha$ then $f_{gp} = -\infty$. In all remaining cases, after deleting the columns of the array $(a_{\alpha,i})$ corresponding to the indices $i$ such that $f_{2d,i}=0$, we are reduced to the case where $f_{2d,i}>0$ for all $i$, i.e., we can apply geometric programming to compute $f_{gp}$.

\medskip

A special case occurs when $f_{2d,i}>0$, for $i=1,\dots,n$ and $\{\alpha\in\Delta~:~|\alpha|=2d\}=\emptyset$. In this case, the equality constraints in the computation of $m^*$ are vacuous
and the feasibility set is always non-empty, so $f_{gp} \ne -\infty$.
\begin{crl}\label{fgpinterior}
If $|\alpha|<2d$ for each $\alpha\in\Delta$ and $f_{2d,i}>0$ for $i=1,\dots,n$, then $f_{gp}\ne -\infty$ and $f_{gp}=f_0-m^*$ where
$m^*$ is the output of the geometric program
\[
\left\lbrace
\begin{array}{ll}
	\textrm{Minimize} & \sum_{\alpha\in\Delta}(2d-|\alpha|)\left[\left(\frac{f_{\alpha}}{2d}\right)^{2d}\alpha^{\alpha}a_{\alpha}^{-\alpha}\right]^{\frac{1}{2d-|\alpha|}} \\
	\textrm{Subject to} & \sum_{\alpha\in\Delta}a_{\alpha,i}\le f_{2d,i}, \quad i=1,\cdots,n.\\
\end{array}
\right.
\]
\end{crl}
\begin{proof} Immediate from Corollary \ref{fgprog}.
\end{proof}

\begin{exm}
(1) Let $f$ be the polynomial of Example \ref{deg6}. Then $f_{gp}=f_{sos}=f_*\approx 0.3265$.\\
(2) For $g(X,Y,Z)=X^6+Y^6+Z^6+X^2YZ^2-X^4-Y^4-Z^4-YZ^3-XY^2+2$, $g_*\approx0.667$, and $g_{gp}=g_{sos}\approx -1.6728$.\\
(3) For $h(X,Y,Z)=g(X,Y,Z)+X^2$, we have $h_{gp}\approx-1.6728 < h_{sos}\approx-0.5028$ and $h_*\approx0.839$.\\
\end{exm}
To compare the running time efficiency of computation of $f_{sos}$ using semidefinite programming with computation of $f_{gp}$ using geometric programming, we set up a test to keep track of the
running times. All the polynomials were taken randomly of the form $X_1^{2d}+\cdots+X_n^{2d}+g(\ux)$ where $g\in\rx$ is of degree
$\leq 2d-1$. In each case the computation is done for 50 polynomial with coefficients uniformly distributed on a certain symmetric
interval, using \textsc{SosTools} and \textsc{GPposy} for \textsc{Matlab}\footnote{\textbf{Hardware and Software specifications.}
Processor: Intel\textregistered~ Core\texttrademark2 Duo
CPU P8400 @ 2.26GHz, Memory: 2 GB, OS: Ubuntu 10.04-32 bit, \textsc{Matlab}: 7.9.0.529 (R2009b)}.

Although, sometimes there is a large gap between $f_{sos}$
and $f_{gp}$, the running time tables show that computation of $f_{gp}$ is much faster than $f_{sos}$.
\begin{table}[ht]
\caption{Average running time (seconds) to calculate $f_{sos}$}
\centering
\begin{tabular}{|c|ccccc|}
\hline
n$\backslash$2d & 4 & 6 & 8 & 10 & 12\\
\hline
3 & 0.73 & 1 & 1.66 & 2.9 & 6.38\\
4 & 0.98 & 1.8 & 5.7 & 25.5 & - \\
5 & 1.43 & 4.13 & 44.6 & - & -\\
6 & 1.59 & 13.24 & 573 & - & -\\
\hline
\end{tabular}
\end{table}
\begin{table}[ht]
\caption{Average running time (seconds) to calculate $f_{gp}$}
\centering
\begin{tabular}{|c|ccccc|}
\hline
n$\backslash$2d & 4 & 6 & 8 & 10 & 12\\
\hline
3 & 0.08 & 0.08 & 0.12 & 0.28 & 0.36\\
4 & 0.08 & 0.13 & 0.3 & 0.76 & 2.23 \\
5 & 0.08 & 0.25 & 0.8 & 3.42 & -\\
6 & 0.09 & 0.37 & 2.2 & - & -\\
\hline
\end{tabular}
\end{table}
\begin{exm}
Let $f(X,Y,Z)=X^{40}+Y^{40}+Z^{40}-XYZ$.
According to Remark \ref{singleterm}, $f_*=f_{sos}=f_{gp}$. The running time for
computing $f_{gp} \approx-0.686$ using geometric programming was $0.18$ seconds, but when we attempted to compute $f_{sos}$ directly, using semidefinite programming, the machine ran out of memory and halted, after about 4 hours.
\end{exm}
%
\section{Explicit lower bounds}\label{ExpBnds}
%
We explain how the lower bounds for $f$ established in \cite[Section 3]{MGHMM} can be obtained by evaluating the objective function of the geometric program in Corollary \ref{fgpinterior} at suitably chosen feasible points.

Recall that for a (univariate) polynomial of the form $p(t)=t^n- \sum_{i=0}^{n-1} a_it^i$, where each $a_i$ is nonnegative and at least one $a_i$
is nonzero, $C(p)$ denotes the unique positive root of $p$ \cite[Theorem 1.1.3]{VVP}.
See \cite{D}, \cite[Ex. 4.6.2: 20]{K} or \cite[Proposition 1.2]{MGHMM} for more details.
\begin{crl}\label{rl}
If $|\alpha|<2d$ for each $\alpha\in\Delta$ and $f_{2d,i}>0$ for $i=1,\dots,n$, then $f_{gp}\ge r_L$, where
\[
\begin{array}{lcl}
	r_L & := & f_0-\frac{1}{2d}\sum_{\alpha\in\Delta}(2d-|\alpha|)|f_{\alpha}|k^{|\alpha|}(f_{2d}^{-\alpha})^{\frac{1}{2d}}\\
	k & \ge & \max\limits_{i=1,\cdots,n} C(t^{2d}-\frac{1}{2d}\sum_{\alpha\in\Delta}\alpha_i|f_{\alpha}|f_{2d,i}^{-\frac{|\alpha|}{2d}}t^{|\alpha|}).
\end{array}
\]
\end{crl}
Here, $f_{2d}^{-\alpha}:=\prod_{i=1}^nf_{2d,i}^{-\alpha_i}$.

\begin{proof}
For each $\alpha\in\Delta$ and
$i=1,\cdots,n$. Let
\[a_{\alpha,i}=\frac{\alpha_i}{2dk^{2d-|\alpha|}}|f_{\alpha}|(f_{2d,i})^{1-\frac{|\alpha|}{2d}}.\]
By definition of $k$, for each $i$, $\frac{1}{2d}\sum_{\alpha\in\Delta}\alpha_i|f_{\alpha}|(f_{2d,i})^{-\frac{|\alpha|}{2d}}k^{|\alpha|}\leq k^{2d}$, hence
\[
\sum_{\alpha\in\Delta}a_{\alpha,i} = \sum_{\alpha\in\Delta}\frac{\alpha_i}{2dk^{2d-|\alpha|}}|f_{\alpha}|(f_{2d,i})^{1-\frac{|\alpha|}{2d}}\leq f_{2d,i}.
\]
This shows that the array $(a_{\alpha,i}:\alpha\in\Delta,i=1,\cdots,n)$ is a feasible point for the geometric program in the statement of Corollary \ref{fgpinterior}.
Plugging this into the objective function of the program  yields
\[
\begin{array}{rl}
 & \sum_{\alpha\in\Delta}(2d-|\alpha|)\left[\left(\frac{f_{\alpha}}{2d}\right)^{2d}\prod_{\alpha_i\neq0}\left(\frac{\alpha_i}{a_{\alpha,i}}\right)^{\alpha_i}\right]^{\frac{1}{2d-|\alpha|}}\\
= & \sum_{\alpha\in\Delta}(2d-|\alpha|)\left[\left(\frac{f_{\alpha}}{2d}\right)^{2d}\prod_{\alpha_i\neq0}\left(\frac{2d\alpha_i}{\alpha_i|f_{\alpha}|k^{|\alpha|-2d}}(f_{2d,i})^{\frac{|\alpha|}{2d}-1}\right)^{\alpha_i}\right]^{\frac{1}{2d-|\alpha|}}\\
= & \sum_{\alpha\in\Delta}(2d-|\alpha|)\left[\left(\frac{f_{\alpha}}{2d}\right)^{2d}\prod_{\alpha_i\neq0}\left(\frac{2d}{|f_{\alpha}|}k^{2d-|\alpha|}(f_{2d,i})^{\frac{|\alpha|-2d}{2d}}\right)^{\alpha_i}\right]^{\frac{1}{2d-|\alpha|}}\\
= & \frac{1}{2d}\sum_{\alpha\in\Delta}(2d-|\alpha|)|f_{\alpha}|k^{|\alpha|}(f_{2d}^{-\alpha})^{\frac{1}{2d}},
\end{array}
\]
so $r_L=f_0-\frac{1}{2d}\sum_{\alpha\in\Delta}(2d-|\alpha|)|f_{\alpha}|k^{|\alpha|}(f_{2d}^{-\alpha})^{\frac{1}{2d}}\leq f_{gp}$.
\end{proof}
\begin{crl}
	\label{rfk}
If $|\alpha|<2d$ for each $\alpha\in\Delta$ and $f_{2d,i}>0$ for $i=1,\dots,n$, then $f_{gp}\geq r_{FK}$, where $r_{FK} :=f_0-k^{2d}$,
$k \ge C(t^{2d}-\sum_{i=1}^{2d-1}b_it^i)$,
$$b_i :=\frac{1}{2d}(2d-i)^{\frac{2d-i}{2d}}\sum_{\alpha\in\Delta, |\alpha|=i}|f_{\alpha}|(\alpha^{\alpha}f_{2d}^{-\alpha})^{\frac{1}{2d}}, \ i=1,\dots, 2d-1.$$
\end{crl}
\begin{proof}
Define
\[
	 a_{\alpha,i}:=(2d-|\alpha|)^{\frac{2d-|\alpha|}{2d}}\frac{|f_{\alpha}|}{2d}(\alpha^{\alpha}f_{2d}^{-\alpha})^{1/2d}f_{2d,i}k^{|\alpha|-2d}.
\]
Note that $\sum_{i=1}^{2d-1}b_ik^i\leq k^{2d}$ and, for each $i=1,\ldots,n$,
\[
\begin{array}{lcl}
	\displaystyle{\sum_{\alpha\in\Delta}a_{\alpha,i}} & = & \displaystyle{\sum_{\alpha\in\Delta}(2d-|\alpha|)^{\frac{2d-|\alpha|}{2d}}\frac{|f_{\alpha}|}{2d}(\alpha^{\alpha}f_{2d}^{-\alpha})^{\frac{1}{2d}}f_{2d,i}k^{|\alpha|-2d}}\\
	& = & \displaystyle{\sum_{j=1}^{2d-1}\sum_{\alpha\in\Delta,|\alpha|=j}(2d-j)^{\frac{2d-j}{2d}}\frac{|f_{\alpha}|}{2d}(\alpha^{\alpha}f_{2d}^{-\alpha})^{\frac{1}{2d}}f_{2d,i}k^{j-2d}}\\
	& = & \displaystyle{f_{2d,i}\sum_{j=1}^{2d-1}\frac{1}{2d}k^{-2d}k^j(2d-j)^{\frac{2d-j}{2d}}\sum_{\alpha\in\Delta,|\alpha|=j}|f_{\alpha}|(\alpha^{\alpha}f_{2d}^{-\alpha})^{\frac{1}{2d}}}\\
	& = & \displaystyle{f_{2d,i}k^{-2d}\sum_{j=1}^{2d-1}b_jk^j}\\
	& \leq & f_{2d,i}.
\end{array}
\]
Hence, $(a_{\alpha,i}:\alpha\in\Delta, i=1,\cdots,n)$ belongs to the feasible set of the geometric program in Corollary \ref{fgpinterior}. Plugging into the objective function, one sees after some effort that
$$ \sum_{\alpha\in\Delta}(2d-|\alpha|)\left[\left(\frac{f_{\alpha}}{2d}\right)^{2d}\alpha^{\alpha}a_{\alpha}^{-\alpha}\right]^{\frac{1}{2d-|\alpha|}} = \sum_{j=1}^{2d-1} b_jk^j \le k^n,$$
so $r_{FK}\leq f_{sos}$.
\end{proof}
\begin{crl}	\label{rdmt}
If $|\alpha|<2d$ for each $\alpha\in\Delta$ and $f_{2d,i}>0$ for $i=1,\dots,n$, then
\[
f_{gp}\geq r_{dmt}:=f_0-\sum_{\alpha\in\Delta}(2d-|\alpha|)\left[\left(\frac{f_{\alpha}}{2d}\right)^{2d}t^{|\alpha|}\alpha^{\alpha}f_{2d}^{-\alpha}\right]^{\frac{1}{2d-|\alpha|}},
\]
where $t:=|\Delta|$.
\end{crl}
\begin{proof}
Take $a_{\alpha,i}=\frac{f_{2d,i}}{t}$ and apply Corollary \ref{fgpinterior}.
\end{proof}
\begin{rem}
Let $C$ be a cone in a finite dimensional real vector space $V$. Let $C^{\circ}$ denote the interior of $C$. If $a \in C^{\circ}$ and
$b\in V$ then $b\in C^{\circ}$ if and only if $b-\epsilon a \in C$ for some real $\epsilon>0$ (See \cite[Lemma 6.1.3]{M1} or \cite[Remark 2.6]{MGHMM}).
Since $\sum_{i=1}^nX_i^{2d}\in\Sigma_{2d,n}^{\circ}$ \cite[Corollary 2.5]{MGHMM}, for a polynomial $f$ of degree $2d$,
with $f_{2d}\in\Sigma_{2d,n}^{\circ}$, there exists an $\epsilon>0$ such that $g=f_{2d}-\epsilon(\sum_{i=1}^nX_i^{2d})\in\Sigma_{2d,n}$.
The hypothesis of Corollary \ref{fgpinterior} holds for $f-g$. In this way, corollaries \ref{fgpinterior}, \ref{rl}, \ref{rfk} and \ref{rdmt}, provide lower bounds
for $f_{sos}$. Moreover, the lower bounds obtained in this way, using corollaries \ref{rl}, \ref{rfk} and \ref{rdmt}, are exactly the lower bounds obtained in \cite{MGHMM}.
\end{rem}
The bounds $r_L$, $r_{FK}$, $r_{dmt}$ provided by corollaries \ref{rl}, \ref{rfk} and \ref{rdmt} are typically not as good as the bound $f_{gp}$ provided by Corollary \ref{fgpinterior}
.
\begin{exm} (Compare to \cite[Example 4.2]{MGHMM})\\
(a) For $f(X,Y) = X^6+Y^6+7XY-2X^2+7$, we have $r_L\approx-1.124$, $r_{FK} \approx-0.99$, $r_{dmt}\approx-1.67$ and $f_{sos}=f_{gp}\approx-0.4464$,
so $f_{gp}>r_{FK}>r_L>r_{dmt}$.\\
(b) For $f(X,Y) = X^6+Y^6+4XY+10Y+13$, $r_L \approx -0.81$, $r_{FK} \approx -0.93$, $r_{dmt} \approx -0.69$ and $f_{gp}\approx0.15\approx f_{sos}$,
so $f_{gp}>r_{dmt} > r_L> r_{FK}$.\\
(c) For $f(X,Y)= X^4+Y^4+XY-X^2-Y^2+1$, $f_{sos}=f_{gp}=r_L = -0.125$, $r_{FK} \approx -0.832$ and $r_{dmt} \approx -0.875$, so $f_{gp}=r_L>r_{FK} > r_{dmt}$.
\end{exm}
%

%
\end{document}